\documentclass[reqno, 11pt]{amsart}
\usepackage{graphicx, enumerate, enumitem, amsmath, amssymb,amsthm,mathabx,manfnt,hyperref,braket,amscd,
mathrsfs,tikz}
\usepackage[utf8]{inputenc}

\newtheorem{thm}{Theorem}
\newtheorem{prop}{Proposition}[section]
\newtheorem{lem}[prop]{Lemma}

\theoremstyle{remark}

\numberwithin{equation}{section}

\newcommand{\cx}{{\mathbb{C}}}

\newcommand{\D}{\mathbb D}
\newcommand{\C}{\mathbb C}
\newcommand{\B}{\mathbb B}
\renewcommand{\P}{\mathbb P}

\newcommand{\N}{\mathbb{N}}

\newcommand{\wt}{\widetilde}

 \DeclareMathOperator{\id}{id}

\title[Symmetric powers of balls]{Proper holomorphic self-maps of symmetric powers of balls}

\author{Debraj Chakrabarti}
\address{Department of Mathematics, Central Michigan University, Mt. Pleasant,  MI 48859,  USA}
\email{chakr2d@cmich.edu}
\author{Christopher Grow}
\address{Department of Mathematics, Central Michigan University, Mt. Pleasant,  MI 48859,  USA }
\email{grow1cm@cmich.edu}
\thanks{Debraj Chakrabarti was partially supported by a grant from the NSF (\#1600371), a collaboration grant from the
Simons Foundation (\# 316632), and  an 
Early Career internal grant from Central Michigan University. Christopher Grow was partially supported by a Research Assistantship from the Central Michigan University mathematics department.}
\begin{document}
\begin{abstract} We show that each proper holomorphic self map of a symmetric power of the  unit ball is an automorphism 
naturally induced by an automorphism of the unit ball, provided the ball is of dimension at least two.
\end{abstract}
\maketitle
 \section{Introduction}
Let $\D^m$ denote the $m$-dimensional polydisc in $\C^m$, and for $1\leq k \leq m$, denote by $\sigma_k$ the $k$-th elementary symmetric polynomial in $m$ variables.  The subset of $\cx^m$ given by
\begin{equation}\label{eq-sigmamd}
\Sigma^m \D =\left \{(\sigma_1(z),\sigma_2(z),\dots,\sigma_m(z))\in \cx^m| 
z\in \D^m\right\}
\end{equation}
is known as the {\em symmetrized polydisc} of $m$-dimensions. 
It turns out that $\Sigma^m \D$ is a pseudoconvex domain $\cx^n$ with remarkable function theoretic properties, 
and applications to engineering (see \cite{ agleryoung1} and the work inspired by it.) 
Of particular interest are the symmetries and mapping properties of these domains. In \cite{jarpfl}, Jarnicki and Pflug determined the biholomorphic automorphisms of $\Sigma^m\D$. More generally, we have the following result of Edigarian and Zwonek on proper self-maps of $\Sigma^m \D$:
\begin{thm}[See \cite{edigarian1, edigarian2}]
\label{thm-ez}
Let $f:\Sigma^m\D\to\Sigma^m\D$ be a proper holomorphic map, and let $\sigma=(\sigma_1,\dots,\sigma_m):\C^m\to\C^m$ be as in \eqref{eq-sigmamd}. Then, there exists a proper holomorphic map $B:\D\to\D$ such that 
\begin{equation*}
    f(\sigma(z^1,\dots,z^m)=\sigma(B(z^1),\dots,B(z^m)).
\end{equation*}
\end{thm}

Recall that a proper holomorphic map $B:\D\to \D$ is represented by a finite Blaschke product, so the above result gives a 
complete characterization of proper holomorphic self-maps of $\Sigma^m \D$, so that each such map is induced by a proper holomorphic self-map of the disc. 

In this note we prove an analogous result for proper self-maps of complex analytic spaces analogous to the symmetrized 
polydisc, where the disc is replaced by a higher dimensional ball. To define these spaces, let $\B_s\subset \cx^s$ denote the unit ball in $\cx^s$, and for some positive integer $m$, let $(\B_s)^m_{\rm Sym}$ denote the $m$-fold {\em symmetric power} of  $\B_s$, i.e., the collection of all {\em unordered} $m$-tuples $\langle z^1,z^2,\dots, z^m\rangle$, where each $z^j\in \B_s$.  Note that  the construction of the symmetric power is functorial: given any map $g:\B_s\to \B_s$, there is a naturally defined map $g^m_{\rm Sym}:(\B_s)^m_{\rm Sym}\to (\B_s)^m_{\rm Sym}$ given by
\[ g^m_{\rm Sym} (\langle z^1, z^2,\dots, z^m\rangle)= \langle g(z^1), g(z^2),\dots, f(z^m)\rangle,\]
which is easily seen to be well-defined.  See below in Section~\ref{sec-sympowers} for a discussion of this notion, and further details. Then, 
$(\B_s)^m_{\rm Sym}$ is a complex analytic space, and $\Sigma^m \D$ is biholomorphic to $\D^m_{\rm Sym}$. The main result of this paper is:
\begin{thm}
\label{thm-main}
Let $s\geq 2$, $m\geq2$, and let $f:(\B_s)^m_{\rm{Sym}}\to(\B_s)^m_{\rm{Sym}}$ be a proper holomorphic map. Then, there exists a holomorphic automorphism $g:\B_s\to \B_s$ such that 
\[ f= g^m_{\rm Sym},\]
that is $f$ is the $m$-th symmetric power of $g$. 
\end{thm}
It follows in particular that each proper self-map of $(\B_s)^m_{\rm{Sym}}$ is in fact an automorphism. Recall also (see \cite[p. 25]{rudin} that an automorphism of $\B_s$ is of the form \begin{equation}
\label{autB}
    \varphi_a(z)=\frac{a-P_a z-s_a Q_a z}{1-\langle z,a\rangle},
\end{equation}
where $a\in\B_n$, $P_a$ is the orthogonal projection from $\C^n$ onto the one dimensional complex linear subspace spanned by $a$, $Q_a=\id-P_a$ is the orthogonal projection from $\C^n$ onto the orthogonal complement of the one dimensional complex linear subspace spanned by $a$, and $s_a=(1-|a|^2)^\frac{1}{2}$. 

Note also that the domain $\B_s$ may be replaced in Theorem~\ref{thm-main} by any strongly pseudoconvex domain, without any change in the proof.
We prefer to state it in this special case for simplicity.

\section{Symmetric Powers} \label{sec-sympowers}
\subsection{Complex Symmetric powers} Recall that, informally, 
a {\em complex analytic space}  is made of {\em local analytic subsets} of $\cx^n$ glued together analytically, just as  a complex manifold of dimension $s$ is made of open sets of $\cx^s$ analytically glued together (see \cite{chirka, whitney, gunros} for more information).  Recall also that an analytic subset of $\cx^n$ is given near each of point of $\cx^n$ by the vanishing of a family of analytic functions, and a local analytic subset of $\cx^n$ is an open subset of an analytic subset of $\cx^n$.

Let $X$ be a complex manifold, or more generally, a complex analytic space. let $X^m$ denote the $m$-th Cartesian power of $X$, which is by definition the collection of
ordered tuples
\[ \{ (x^1,\dots, x^m), x^j \in X, j=1,\dots,m\}.\]
$X^m$ is then a complex manifold in an obvious way. 
The symmetric group $S_m$ of bijective mappings of the set $\{1,\dots, m\}$ acts on $X^m$ in as biholomorphic automorphisms: for $\sigma\in S_m$, and $x= (x^1,\dots, x^m)\in X^m$ we set
\[ \sigma\cdot x= (x^{\sigma(1)}, \dots, x^{\sigma(m)}).\]
The $m$-th {\em symmetric Power} of $X$, denoted by $X^m_{\rm Sym}$ is the quotient of  $X^m$ under the action of $S_m$ defined above, i.e, points of $X^m_{\rm Sym}$ are orbits of the action of $S_m$ on $X^m$. We denote by
\[ \pi: X^m \to X^m_{\rm Sym}\]
the natural quotient map.  It follows from the general theory of complex analytic spaces that $X^m_{\rm Sym}$ has a canonical structure of an analytic space, i.e., the quotient analytic space of $X^m$ under the action of $S_m$ as biholomorphic automorphisms (see \cite{cartan1, cartan2} and \cite{gunros}).  When $X^m_{\rm Sym}$ is given
this complex structure, the map $\pi$ becomes a proper holomorphic map. 

In our application, we are interested in the case when $X=\B_s$, the unit ball in $\cx^s$. The symmetric power $(\B_s)^m_{\rm Sym}$ is then biholomorphic to a local analytic set, in fact to an open subset of a certain affine algebraic 
variety in $\cx^N$ for some large $N$ depending on $m$ and $s$.  Though logically not needed for the proof of Theorem~\ref{thm-main}, we give a short account of this construction in order to explain the relation of Theorem~\ref{thm-main} with Theorem~\ref{thm-ez}, as well as to emphasize the elementary nature of the constructions.

\subsection{Embedding of symmetric powers of projective spaces}
Recall that the $m$-fold \emph{symmetric tensor product} of a (complex) vector space $V$, denoted $V^{\odot m}$ is defined as $V^{\odot m}=V^{\otimes m}\slash\sim$, where $V^{\otimes m}=V\otimes V\otimes\dots\otimes V$ is the $m$-fold tensor product of $V$ with itself, and the equivalence relation $\sim$ is defined by:
\begin{equation}
\label{symtensorrel}
    v^{\sigma(1)}\otimes v^{\sigma(2)}\otimes\dots\otimes v^{\sigma(m)}\sim v^1\otimes v^2\otimes\dots\otimes v^m
\end{equation}
for all $v^1,v^2,\dots,v^m\in V$ and all $\sigma\in S_m$. We denote the equivalence class of  $v^1\otimes v^2\otimes\dots\otimes v^m$ under this equivalence relation (which is an element of 
$ V^{\odot m}$) by $ v^1\odot v^2\odot\dots\odot v^m.$
Just as the tensor product of vector spaces is itself a vector space, the symmetric tensor product $V^{\odot m}$ of a vector space is again a vector space, with a natural linear structure as a quotient vector space of $V^{\otimes m}$. Suppose $V$ is finite-dimensional of dimension $s+1$, and let $\{e_0,\dots,e_s\}$ be a basis of $V$. Let $\mu=(\mu_0,\mu_2,\dots,\mu_s)\in\N^{s+1}$ be a multi-index with $|\mu|=\sum_{j=0}^s \mu_j =m$, and let $\mathbf{e}_\mu$ denote the element of $V^{\odot m}$ given by
\begin{equation}\label{eq-emu}
    \mathbf{e}_\mu=e_{i_1}\odot e_{i_2}\odot\dots\odot e_{i_m}
\end{equation}
where exactly $\mu_j$ of the $e_{i_k}$ are equal to $e_j$. It is not difficult to see that  $\{\mathbf{e}_\mu\}_{|\mu|=m}$ gives a basis for $V^{\odot m}$. Therefore the dimension of $V^{\odot m}$ is the number of solutions of $\sum_{j=0}^s \mu_j=m$, i.e.
\begin{equation}\label{dimension}
\dim V^{\odot m}=\binom{m+s}{m}.
\end{equation}
Given a vector space $V$, there is a natural mapping from the symmetric power, $V^m_{\rm Sym}$ (which is just a set), to the symmetric tensor product $V^{\odot m}$ (which is a vector space), given by
\begin{equation}
\label{symtotenssym}
    \langle v^1,v^2,\dots,v^m\rangle\mapsto v^1\odot v^2\odot\dots\odot v^m,
\end{equation}
which is easily seen to be well-defined.

Denote by $\P(V)$ the  {projectivization} of the vector space $V$, and for $v\in V$, denote by 
$ [v]$  its equivalence class in the projectivization  $ \P(V)$. There is a natural map 

\[ \psi:(\P(V))^m_{\rm Sym}\to \P(V^{\odot m})\]
 induced by the map \eqref{symtotenssym}  given by
\begin{equation}
\label{segre-whitney}
    \psi(\langle[v^1],[v^2],\dots,[v^m]\rangle)=  \left[v^1\odot v^2\odot\dots\odot v^m\right]\in \P(V^{\odot m}),
\end{equation}
with $v^j\in V$.
We call this the \emph{Segre-Whitney} map (see  Appendix  V of \cite{whitney}). It is easily seen to be well-defined and injective, and  is a symmetric version of the well-known classical {\em Segre embedding} of the product of two or more projective spaces as a projective variety in a higher dimensional projective space.

Thanks to classical results in complex algebraic geometry, the image $\psi\left( \P(V))^m_{\rm Sym}\right)$ is a projective
algebraic variety in the projective space $ \P(V^{\odot m})$, which, if $V=\cx^{s+1}$ is of dimension 
\begin{equation}
\label{Nms}
    N(m,s)= \binom{m+s}{m}-1= \dim_\C\left((\C^{s+1})^{\odot m}\right)-1.
\end{equation}
\subsection{Embedding of $(\B_s)^m_{\rm Sym}$} 
Let $E$ be a subset of $\C\P^s=\P(\cx^{s+1})$, and let $i:E\to \C\P^s$ be the inclusion map. Then we have a natural 
inclusion 
\[ i^m_{\rm Sym}: E^m_{\rm Sym} \to (\C\P^s)^m_{\rm Sym},\]
so that we can think of $E^m_{\rm Sym}$ as a subset of $(\C\P^s)^m_{\rm Sym}$. Composing with the Segre-Whitney embedding $\psi$ of $(\C\P^s)^m_{\rm Sym}$ in $\C\P^{N(m,s)}$, where $N(m,s)$ is defined as in \eqref{Nms}, we obtain an embedding
\[ \psi\circ i^m_{\rm Sym}: E^m_{\rm Sym} \to \C\P^{N(m,s)}\]
which we will again call the {\em Segre-Whitney embedding}. In this way we can think of symmetric powers as sitting in some projective space. We will denote this embedded version of the $m$-th symmetric power of $E$ by $\Sigma^m E$, i.e.,
\begin{equation}\label{sigmamu}
    \Sigma^mE = \psi\circ i^m_{\rm Sym}({E^m_{\rm Sym}}) \subset \C\P^{N(m,s)}.
\end{equation}
Two special cases of this construction are relevant here. The first is when $E$ is an affine piece of $\C\P^s$, so that $E$ can 
be identified with $\cx^{s}$.  Then, explicitly working through the computations, one can verify that $\Sigma^m E$ is an affine algebraic variety in an affine piece of $\C\P^{N(m,s)}$. So we can think of $\Sigma^m \C^s$ as an affine algebraic variety in $\C^{N(m,s)}$. For details of this construction, see \cite[Appendix V]{whitney}. 

Now, if $E=\B_s$ is the unit ball in an affine piece of $\C\P^s$, then it is easy to see that $\Sigma^m E= \Sigma^m \B_s$ 
is an open subset of the affine algebraic variety $\Sigma^m \C^s$. In this way,  $(\B_s)^m_{\rm Sym}$ is realized as a 
local analytic set, i.e., an open set of an analytic subset  of $\C^{N(m,a)}$ (actually an open set of an affine algebraic variety).

\subsection{The case $s=1$} When $s=1$, we have $N(m,s)=m$, and it is not difficult to see that the Segre-Whitney map
$ \cx^m_{\rm Sym} \to  \cx^m$
is actually a biholomorphism, given by
\[  z \mapsto (\sigma_1(z),\dots, \sigma_m(z)),\]
where $\sigma_k(\langle z_1,\dots, z_m\rangle )$ is the $k$-th elementary symmetric polynomial in the variables $(z_1,\dots, z_m)$.  Then the image of  $\D^m_{\rm Sym}$ is a  pseudoconvex domain $\Sigma^m \D$ in $\cx^m$, called the 
{\em symmetrized polydisc}. See \cite{chakgorai} for more details.  Consequently, the symmetric power $\D^m_{\rm Sym}$ is biholomorphically identified with the domain  $\Sigma^m \D$ in $\cx^m$, which shows that Theorem~\ref{thm-main} is indeed an extension of Theorem~\ref{thm-ez}. 

\section{Proper mappings of Cartesian to Symmetric powers}
The first step in the proof of Theorem~\ref{thm-main} is the following result, which is interesting in its own right:

\begin{thm}
\label{carttosym}
Let $s\geq 2$, $m\geq 2$, and let $f:(\B_s)^m\to(\B_s)^m_{\rm{Sym}}$ be a proper holomorphic map. Then, there exists a proper holomorphic map $\widetilde{f}:(\B_s)^m\to (\B_s)^m$ such that $f=\pi\circ \widetilde{f}$.
\end{thm}
In other words, the map $f$ can be {\em lifted} to a proper holomorphic map
$\widetilde{f}$ such that the following diagram commutes:

\begin{center}
\begin{tikzpicture}[scale=1]
    \node at (0,0) (left) {$(\B_s)^m$};
    \node at (3,0) (right) {$(\B_s)^m_{\rm{Sym}}$};
    \node at (3,2) (top) {$(\B_s)^m$};
    \draw[->] (left) -- (right);
    \draw[->] (top) -- (right);
    \draw[->] (left) -- (top);
    \node[label={south:$f$}] at (1.5,0) {};
    \node[label={east:$\pi$}] at (3,1) {};
    \node[label={[label distance=-5]north west:$\wt{f}$}] at (1.5,1) {};
\end{tikzpicture}
\end{center}

Note however, that the map $\pi$ is {\em not} a covering map, so that the classical theory of lifting maps into a covering space is 
not directly applicable. However, as we will see, we can reduce this problem to a problem involving covering maps by removing the ramification locus and the branching locus of the map $\pi$ from $(\B_s)^m$ and $(\B_s)^m_{\rm Sym}$ respectively. 
\subsection{Fundamental group of complements of analytic sets}
 The proof of Theorem~\ref{carttosym} will involve the computation of some fundamental groups, for which we need the following fact. We include a proof for completeness. 
\begin{prop}
\label{simplyconnected2}
Let $M$ be a connected complex manifold without boundary, $A\subset M$ be an analytic subset, $x$ a point in $M\setminus A$, and $i:M\setminus A\to M$ the inclusion map. Then if
    the complex codimension of $A$ is at least 2, then the homomorphism of the fundamental groups
    \begin{equation*}
        i_*:\pi_1(M\setminus A, x)\to\pi_1(M, x)
    \end{equation*}
    is an isomorphism.
\end{prop}

\begin{proof}

We begin by recalling the following fact from Differential Topology: {\em Let $X$ be a connected $\mathcal{C}^\infty$ differentiable manifold without boundary, $Y\subset M$ be closed submanifold, $x$ a point in $X\setminus Y$, and $i:X\setminus Y\to X$ the inclusion map. Then  if the codimension of $Y$ is at least 3, then the group homomorphism 
    \begin{equation*}
        i_*:\pi_1(X\setminus Y,x)\to\pi_1(X,x)
    \end{equation*}
    is an isomorphism.
    }

For a proof, see \cite[Théorème~2.3, page 146]{godbillon}. Essentially, this is a reflection of the fact that thanks to the low
codimension of $Y$, by a standard transversality argument, there is no problem in homotopically deforming a loop based at $x$ to 
a loop based at $x$ and not intersecting $Y$, and further,  given two loops based at $x$ homotopic in $X$, there is no problem in homotopically deforming them to each other in $X\setminus Y$. 

Now, since $A$ is an analytic subset of a complex manifold $M$, there is a stratification of $A$ by local analytic subsets (see \cite{chirka} for details). More precisely, there exist  pairwise disjoint local analytic subsets $A_j$ of $A$ such that
\begin{equation*}
    A=\bigcup_{i=1}^n A_i,
\end{equation*}
where the set $B_k=\bigcup_{i=1}^k A_i$ is an analytic subset of $M$ and $A_k$ is a closed submanifold of $M\setminus B_{k-1}$ for each $1\leq k\leq n$. Note that, since $A$ has codimension at least $2$ in $M$, then $A_k$ also has codimension at least $2$ in $M\setminus B_{k-1}$ for each $1\leq k\leq n$. Thus, assuming $A\neq M$,
\begin{equation*}
    M_k=M\setminus B_k
\end{equation*}
is a open submanifold of $M$ for each $1\leq k\leq n$. Hence, since
\begin{equation*}
    M_k=M_{k-1}\setminus A_k,
\end{equation*}
each inclusion in the following chain
\begin{center}
\begin{tikzpicture}[scale=1]
    \node at (-1,0) (1) {$M\setminus A=M_n$};
    \node at (2,0) (2) {$M_{n-1}$};
    \node at (4,0) (3) {$\dots$};
    \node at (6,0) (4) {$M_1$};
    \node at (8.5,0) (5) {$M_0=M$};
    \draw[->] (1) -- (2);
    \draw[->] (2) -- (3);
    \draw[->] (3) -- (4);
    \draw[->] (4) -- (5);
    \node[label={north:$i$}] at (.75,0) {};
    \node[label={north:$i$}] at (3,0) {};
    \node[label={north:$i$}] at (5,0) {};
    \node[label={north:$i$}] at (7,0) {};
\end{tikzpicture}
\end{center}
is an isomorphism of groups by the fact from Differential Topology quoted in the first paragraph, since we may take $X=M_k$ and $Y=A_k$ and the conditions are satisfied.   The conclusion follows.
\end{proof}
\subsection{Branching behavior of $\pi$}
Since $\pi$ is proper holomorphic map of equidimensional complex analytic sets, it follows that $\pi$ must be a covering map when the analytic sets over which it is branched are removed from the source and the target. However, given the 
elementary nature of the considerations here, one can be much more explicit in this special case.  
Let $m_1,\dots,m_k$ be a partition of $m$, i.e., $m_j$ be positive integers such that $\sum_{j=1}^k m_j=m$. 
We denote  by 
\begin{equation}\label{eq-whitney}
 \langle x^1:m_1, x^2:m_2, \dots, x^k:m_k\rangle   
\end{equation}
the element of $(\B_s)^m_{\rm Sym}$ in which $x^j$ is repeated $m_j$ times.  Let $V(m_1,\dots, m_k)$ be the set of points  $\alpha$ in ${\B_s}^m_{\rm Sym}$ such that there are {\em distinct} $x^1,\dots,x^k\in {\B_s}$ such that 
\begin{equation*}
    \alpha= \langle x^1:m_1, \dots, x^k: m_k\rangle,
\end{equation*}
where we use the notation \eqref{eq-whitney}, that is $x^j$ is repeated $m_j$ times. Also let 
\begin{equation*}
    \wt{V}(m_1,\dots, m_k)= \pi^{-1}(V(m_1,\dots,m_k)) \subset {\B_s}^m.
\end{equation*}
\begin{prop} \label{prop-covering} The restricted map
\[ \pi: \wt{V}(m_1,\dots, m_k) \to V(m_1,\dots,m_k)\]
is a holomorphic  covering map of degree 
\[ \frac{m!}{m_1!m_2!\cdots m_k!}.\]
\end{prop}
\begin{proof}
Let $\alpha\in V(m_1,\dots,m_k)$. We will show that there is an open set $U_\alpha\subset V(m_1,\dots,m_k)$, containing $\alpha$, such that $\pi^{-1}(U_\alpha)$ is a disjoint union of open subsets $\wt{U}_\alpha^i\subset\wt{V}(m_1,\dots,m_k)$, with the restriction of $\pi$ to each $\wt{U}_\alpha^i$ a homeomorphism onto its image $\pi(\wt{U}_\alpha^i)$.

We can write $\alpha= \langle x^1:m_1, \dots, x^k: m_k\rangle$ for some distinct $x^1,\dots,x^k\in {\B_s}$. Since ${\B_s}$ is Hausdorff, there exist disjoint open subsets $U_i\subset {\B_s}$ with $x^i\in U_j$ if and only if $i=j$. Now, let 
\begin{equation*}
    U_\alpha=\{\langle \alpha^1:m_1,\dots,\alpha^k:m_k\rangle\in V(m_1,\dots,m_k)\mid \alpha^i\in U_i\}.
\end{equation*}    
Note that $U_\alpha$ is open in in the subspace topology defined on $V(m_1,\dots,m_k)$, since
\begin{equation*}
    U_\alpha=\pi(U_1^{m_1}\times\cdots\times U_k^{m_k})\cap V(m_1,\dots,m_k),
\end{equation*}    
where $U_i^{m_i}$ denotes the $m_i$-fold Cartesian power of $U_i$. Let $\sigma(y^1,\dots,y^m)$ denote $(y^{\sigma(1)},\dots,y^{\sigma(m)})$ for $\sigma\in S_m$ and $(y^1,\dots,y^m)\in {\B_s}^m$. Then, $\pi^{-1}(\langle y^1,\dots,y^m\rangle)=\{\sigma(y^1,\dots,y^m)\mid \sigma\in S_m\}$, and we have
\begin{align*}
    \pi^{-1}(U_\alpha)&=\bigcup_{\sigma\in S_m}\sigma\left((U_1^{m_1}\times\cdots\times U_k^{m_k})\cap\wt{V}(m_1,\dots,m_k)\right) \\
    &=\bigcup_{\sigma\in S_m}\left(\sigma(U_1^{m_1}\times\cdots\times U_k^{m_k})\cap\wt{V}(m_1,\dots,m_k)\right) \\
    &=\left(\bigcup_{\sigma\in S_m}\sigma(U_1^{m_1}\times\cdots\times U_k^{m_k})\right)\cap\wt{V}(m_1,\dots,m_k).
\end{align*}
Since the sets $\sigma(U_1^{m_1}\times\cdots\times U_k^{m_k})$ are open in ${\B_s}^m$, it follows that $\pi^{-1}(U_\alpha)$ is open in $\wt{V}(m_a,\dots,m_k)$ with the subspace topology. Also, since $U_i\cap U_j=\emptyset$ for $i\neq j$, the sets $\sigma(U_1^{m_1}\times\cdots\times U_k^{m_k})$ and $\tau(U_1^{m_1}\times\cdots\times U_k^{m_k})$ are either identical or disjoint for each $\sigma,\tau\in S_m$, and the number of distinct sets $\sigma(U_1^{m_1}\times\cdots\times U_k^{m_k})$ is equal to the number of distinct preimages $\pi^{-1}(\langle \alpha^1:m_1,\dots,\alpha^k:m_k\rangle)$, which is exactly
\begin{equation*}
    \frac{m!}{m_1!m_2!\cdots m_k!}.
\end{equation*}

Now, the restricted map $\pi:\sigma(U_1^{m_1}\times\cdots\times U_k^{m_k}) \to \pi(U_1^{m_1}\times\cdots\times U_k^{m_k})$ is one-to-one, and hence a homeomorphism, as $\pi$ is a quotient map. Thus, by restricting $\pi$ to the subspace $\sigma(U_1^{m_1}\times\cdots\times U_k^{m_k})\cap\wt{V}(m_1,\dots,m_k)$, we have $\pi:\sigma(U_1^{m_1}\times\cdots\times U_k^{m_k})\cap\wt{V}(m_1,\dots,m_k)\to U_\alpha$ is a homeomorphism as well.
\end{proof}
\begin{prop}
\label{analyticdimension}
Let $X$ and $Y$ be analytic subsets of $\Omega_1\subset\C^n$ and $\Omega_2\subset\C^m$, respectively, and let  $f:X\to \Omega_2$ be a proper holomorphic map such that $f(X)\subset Y$. Then,
\begin{enumerate}
    \item If $f(X)=Y$, then $\dim X= \dim Y$.
    \item If $Y$ is irreducible and $\dim X= \dim Y$, then $f(X)=Y$.
\end{enumerate}
\end{prop}

\begin{proof}
By Remmert's Theorem, $f(X)$ is an analytic subset of $\Omega_2$, with $\dim f(X)=\dim f$. Since $f$ is proper, we cannot have $\dim f<\dim X$. Otherwise, the Rank Theorem  would imply that for some $y\in f(X)$, $f^{-1}(y)$ is a compact analytic subset of $X$ with positive dimension, which is impossible. Hence, we conclude that $\dim f(X)=\dim X$.

Suppose first $f(X)=Y$. Then, evidently, $\dim X=\dim f(X)=\dim Y$, establishing (1).

 Now suppose $\dim X = \dim Y$ and suppose that $f(X)\neq Y$. Then, by well-known properties of analytic sets, $\overline{Y\setminus f(X)}$ is an analytic set, and since $Y$ is closed in $\Omega_2$, $\overline{Y\setminus f(X)}$ is contained in $Y$. Additionally, since $\dim f(X)=\dim Y$, $\overline{Y\setminus f(X)}$ cannot be all of $Y$. Now, since $Y=f(X)\cup \overline{Y\setminus f(X)}$, $Y$ is reducible. Hence, if $Y$ is irreducible and $\dim X = \dim Y$, then $f(X)=Y$, completing the proof of (2).

\end{proof}

We now prove the following lemma:

\begin{lem}
\label{codimension}
Let $A=\{(z^1,\dots, z^m)\in (\B_s)^m: z^i=z^j \text{ for some $i\neq j$}\}$. 
Then  $A$ is an analytic subset of $(\B_s)^m$ of codimension $s$ and $\pi(A)$ is an analytic subset of $(\B_s)^m_{\rm Sym}$ of codimension $s$. The restricted map
\[ \pi: (\B_s)^m \setminus A \to (\B_s)^m_{\rm Sym} \setminus \pi(A) \]
is a holomorphic covering map of complex manifolds. 
\end{lem}
\begin{proof}
Let $A_{ij}$ be the linear subspace of $(\C^s)^m \cong \C^{sm}$ given by
\begin{equation*}
    A_{ij}= \left\{(z^1,\dots, z^m)| z^i=z^j\right\}.
\end{equation*}
Then $A_{ij}$ is defined by the vanishing of $s$ linearly independent linear functionals $z\mapsto z^i_k- z^j_k$ where $1\leq k\leq s$, and 
consequently $A_{ij}$ is of codimension $s$ in $(\C^s)^m$.  Since
\begin{equation*}
    A = \bigcup_{i<j} \left(A_{ij}\cap (\B_s)^m\right),
\end{equation*}
it now follows that  $A$ is an analytic subset  of codimension $s$ in $(\B^s)^m$. Since the finite-to-one  quotient map $\pi:(\B_s)^m\to(\B_s)^m_{\rm Sym}$ is proper and holomorphic, by Remmert's Theorem, $\pi(A)$ is an analytic subset of $(\B_s)^m_{\rm Sym}$. Since $\pi^{-1}(\pi(A))=A$, it follows that $\pi|_A$ is a proper map $A\to\pi(A)$, and since $\pi$ is also surjective, it is easy to see that we must have $\dim A=\dim \pi(A)$. Since $\dim ((\B_s)^m)=\dim ((\B_s)^m_{\rm Sym})$, it follows that $\pi(A)$ has the same codimension in $(\B_s)^m_{\rm Sym}$ as $A$ has in $(\B_s)^m$, which is $s$.
\end{proof}

Recall that, for a partition $m_1,\dots,m_k$ of $m$, $V(m_1,\dots, m_k)$ is the set of points $\alpha$ in $(\B_s)^m_{\rm Sym}$ such that there are distinct $z^1,\dots,z^k\in \B_s$ such that 
\begin{equation*}
    \alpha= \langle z^1:m_1, \dots, z^k: m_k\rangle,
\end{equation*}
where this notation is as in \eqref{eq-whitney}, i.e., $z^j$ is repeated $m_j$ times.  Let us also set
\begin{equation*} 
    \wt{V}(m_1,\dots, m_k)= \pi^{-1}(V(m_1,\dots,m_k)) \subset (\B_s)^m.
\end{equation*}
Then, $(\B^s)^m\setminus A$ is precisely the set $\wt{V}(m_1,\dots,m_k)$ and $(\B^s)^m_{\rm Sym}\setminus \pi(A)$ is precisely the set $V(m_1,\dots,m_k)$.

We are now ready to prove Theorem \ref{carttosym}.

\begin{proof}[Proof of Theorem~\ref{carttosym}]
Let $f:(\B_s)^m\to(\B_s)^m_{\rm{Sym}}$ be a proper holomorphic map and let $A$ be as defined in Lemma \ref{codimension}. Since $\pi(A)$ is an analytic subset in $(\B_s)^m_{\rm{Sym}}$, $(\B_s)^m_{\rm{Sym}}\setminus\pi(A)$ is an open, connected set.

Since $(\B_s)^m\setminus A=\wt{V}(1,1,\dots,1)$ and $(\B_s)^m_{\rm{Sym}}\setminus\pi(A)=V(1,1,\dots,1)$, by Proposition \ref{prop-covering}, $\pi|_{(\B_s)^m\setminus A}$ is a holomorphic covering $(\B_s)^m\setminus A\to(\B_s)^m_{\rm{Sym}}\setminus\pi(A)$. Since a holomorphic covering is a local biholomorphism, $(\B_s)^m_{\rm{Sym}}\setminus\pi(A)$ is an $sm$-dimensional complex manifold. Moreover, since $\pi(A)$ is an analytic subset of $(\B_s)^m_{\rm{Sym}}$, we have that $(\B_s)^m_{\rm{Sym}}\setminus\pi(A)$ is connected and dense in $(\B_s)^m_{\rm Sym}$. Since $\operatorname{reg} ((\B_s)^m_{\rm Sym})$ is an open subset of $(\B_s)^m_{\rm Sym}$ containing $(\B_s)^m_{\rm{Sym}}\setminus\pi(A)$, $\operatorname{reg} ((\B_s)^m_{\rm Sym})$ must be connected, and hence  $(\B_s)^m_{\rm Sym}$ is an irreducible analytic set, with $\dim ((\B_s)^m_{\rm Sym})=\dim ((\B_s)^m_{\rm Sym}\setminus\pi(A))$. Since $f$ is a proper holomorphic map from $(\B_s)^m$, which is a manifold of dimension $sm$, to $(\B_s)^m_{\rm Sym}$, which is an irreducible analytic set of dimension $sm$, by Theorem \ref{analyticdimension}, $f$ is surjective.

Clearly $f|_{f^{-1}(\pi(A))}$ is a proper holomorphic map from $f^{-1}(\pi(A))$, which is an analytic subset of $(\B_s)^m$, onto $\pi(A)$, and so by Proposition \ref{analyticdimension}, $\dim(f^{-1}(\pi(A))=\dim(\pi(A))$. Since we also have $\dim((\B_s)^m)=\dim((\B_s)^m_{\rm{Sym}})=sm$, and we know from Lemma \ref{codimension} that $\pi(A)$ has codimension at least $s$, $f^{-1}(\pi(A))$ must have codimension at least $s$ in $(\B_s)^m$.

Let $U=(\B_s)^m\setminus f^{-1}(\pi(A))$. Since $A$ and $f^{-1}(\pi(A))$ have complex codimension at least $s\geq 2$, by Proposition \ref{simplyconnected2}, both $(\B_s)^m\setminus A$ and $U$ are simply-connected. Hence, $f|_U$ has a holomorphic lift $\wt{f}|_U:U\to(\B_s)^m\setminus A$, where $f|_U=\pi\circ \wt{f}|_U$. Since $\wt{f}|_U$ is bounded on $U$ and $f^{-1}(\pi(A))$ is an analytic set, by Riemann's Continuation Theorem  $\wt{f}|_U$ extends to a holomorphic function $\wt{f}:(\B_s)^m\to(\B_s)^m$ with $f=\pi\circ \wt{f}$.

\begin{center}
\begin{tikzpicture}[scale=1]
    \node at (0,0) (left) {$U$};
    \node at (4,0) (right) {$(\B_s)^m_{\rm{Sym}}\setminus\pi(A)$};
    \node at (4,2) (top) {$(\B_s)^m\setminus A$};
    \draw[->] (left) -- (right);
    \draw[->] (top) -- (right);
    \draw[->] (left) -- (top);
    \node[label={south:$f|_U$}] at (2,0) {};
    \node[label={east:$\pi$}] at (4,1) {};
    \node[label={[label distance=-5]north west:$\wt{f}|_U$}] at (2,1) {};

    \node at (9,0) (left2) {$(\B_s)^m$};
    \node at (12,0) (right2) {$(\B_s)^m_{\rm{Sym}}$};
    \node at (12,2) (top2) {$(\B_s)^m$};
    \draw[->] (left2) -- (right2);
    \draw[->] (top2) -- (right2);
    \draw[->] (left2) -- (top2);
    \node[label={south:$f$}] at (10.5,0) {};
    \node[label={east:$\pi$}] at (12,1) {};
    \node[label={[label distance=-5]north west:$\wt{f}$}] at (10.5,1) {};
\end{tikzpicture}
\end{center}

It remains to show that $\wt{f}:(\B_s)^m\to (\B_s)^m$ is a proper map. Note that $f=\pi\circ\wt{f}$ is a proper map,
and  $f$ is proper.   If $\wt{f}$ were not proper, one could find a sequence $\{z_n\}_{n=1}^\infty$ with no limit points in $(\B_s)^m$ such that  $\wt{f}(z_n)\to w \in (\B_s)^m$. But this composing with $\pi$, we see that $f$ is not proper, which is a contradiction.  Therefore $\wt{f}$ is proper.
\end{proof}

\section{Proof of Theorem~\ref{thm-main}}
Let $f:(\B_s)^m_{\rm{Sym}}\to(\B_s)^m_{\rm{Sym}}$ be a proper holomorphic map. Then, since $\pi$ is proper and holomorphic, $h=f\circ\pi$ is a proper holomorphic map $(\B_s)^m\to(\B_s)^m_{\rm{Sym}}$. By Theorem \ref{carttosym}, $h$ lifts to a proper holomorphic map $\wt{h}:(\B_s)^m\to(\B_s)^m$ with $h=\pi\circ\wt{h}$. By a classical application of the methods of Remmert and Stein (see \cite[page 76]{narasimhan}), we conclude that 
there exist proper holomorphic self-maps of the ball $\B_s$, $\wt{h}_i$ for $i=1,\dots, m$ and a permutation $\sigma$ of 
 $\{1,2,\dots,m\}$ such that 
$\wt{h}$ has the structure \[ \wt{h}(\tau^1,\tau^2,\dots,\tau^m)=\left(\wt{h}_1(\tau^{\sigma(1)}),\wt{h}_2(\tau^{\sigma(2)}),\dots,\wt{h}_m(\tau^{\sigma(m)})\right).\] 
We get the following commutative diagram:

\begin{center}
\begin{tikzpicture}[scale=1]
    \node at (0,0) (bl) {$(\B_s)^m_{\rm{Sym}}$};
    \node at (3,0) (br) {$(\B_s)^m_{\rm{Sym}}$};
    \node at (0,2) (tl) {$(\B_s)^m$};
    \node at (3,2) (tr) {$(\B_s)^m$};
    \draw[->] (tl) -- (tr);
    \draw[->] (bl) -- (br);
    \draw[->] (tl) -- (br);
    \draw[->] (tl) -- (bl);
    \draw[->] (tr) -- (br);
    \node[label={south:$f$}] at (1.5,0) {};
    \node[label={east:$\pi$}] at (3,1) {};
    \node[label={west:$\pi$}] at (0,1) {};
    \node[label={[label distance=-5]south west:$h$}] at (1.5,1) {};
    \node[label={north:$\wt{h}$}] at (1.5,2) {};
\end{tikzpicture}
\end{center}

Since $f\circ\pi=\pi\circ\wt{h}$, and the left-hand side is invariant under the action of $S_m$ on $(\tau^1,\tau^2,\dots,\tau^m)$, we must have
\begin{equation}
\label{symprod}
    \langle\wt{h}_1(\tau^{\sigma(1)}),\wt{h}_2(\tau^{\sigma(2)}),\dots,\wt{h}_m(\tau^{\sigma(m)})\rangle=\langle\wt{h}_1(\tau^1),\wt{h}_2(\tau^2),\dots,\wt{h}_m(\tau^m)\rangle
\end{equation}
for every $(\tau^1,\tau^2,\dots,\tau^m)\in(\B_s)^m$ and every $\sigma\in S_m$. Such a relation cannot hold unless
there is a self map $g$ of $\B_s$ such that for each $j=1,\dots, m$, we have
\[ \wt{h}_j =g,\]
 since otherwise we could choose a $\sigma$ for which the two sides would be distinct.  
 Now, since $f\circ\pi=h$, we have
\begin{align*}
    f(\langle\tau^1,\tau^2,\dots,\tau^m\rangle)&=\langle g(\tau^1),g(\tau^2),\dots,g(\tau^m)\rangle \\
    &=g^m_{\rm{Sym}}(\langle\tau^1,\tau^2,\dots,\tau^m\rangle)
\end{align*}
Since $f$ is proper and $f= g^m_{\rm{Sym}}$, it follows that $g$ is a proper holomorphic self-map of the ball $\B_s$. 
Thanks to a classical result of Alexander (see \cite{alexander}, and also \cite[page 316]{rudin}), for $s\geq 2$, the only proper holomorphic self-mappings of $\B_s$ are the automorphisms of $\B_s$, and it is known that the automorphisms are given by certain multi-dimensional fractional linear maps (see \cite{rudin}).  Hence, $g$ is an automorphism of $\B_s$ (of the form \eqref{autB} and the proof is complete.



\begin{thebibliography}{12}
\bibitem{agleryoung1} J. Agler and N.J. Young, \textit{ A commutant lifting theorem for a domain in $\C^2$ and spectral interpolation.} J. Funct. Anal. 161 (2), 452--477 (1999)

\bibitem{alexander}
H. Alexander, \textit{Proper holomorphic mappings in $\C^n$}, Indiana Univ. Math. J. 26 (1977), 137--146.

\bibitem{cartan1} H. Cartan, \textit{ Quotient d'un espace analytique par un groupe d'automorphismes.} In Algebraic geometry and topology. pp. 90--102. Princeton University Press, Princeton, N. J. (1957)
\bibitem{cartan2} H. Cartan, \textit{ Quotients of complex analytic spaces. } Contributions to function theory (Internat. Colloq. Function Theory, Bombay, 1960) pp. 1--15 Tata Institute of Fundamental Research, Bombay

\bibitem{chakgorai}
D. Chakrabarti and S. Gorai, \textit{
Function theory and holomorphic maps on symmetric products of planar domains.} 
J. Geom. Anal. 25 (2015), no. 4, 2196--2225. 
\bibitem{chirka}
E. M. Chirka, \textit{Complex Analytic Sets}, Kluwer Academic Publishers, 1989.
\bibitem{edigarian1}
A. Edigarian, \textit{Proper Holomorphic Self-Mappings of the Symmetrized Bidisc}, Annales Polonici mathematici 84 (2004), 181--184.

\bibitem{edigarian2}
A. Edigarian and W. Zwonek, \textit{Geometry of the Symmetrized Polydisc}, Archiv der Mathematik 84 (2005), 364--374.
\bibitem{godbillon}
C. Godbillon, \textit{\'{E}l\'{e}ments du topologie alg\'{e}brique}, Hermann 1971.
\bibitem{gunros} Robert C. Gunning and  Hugo Rossi;
{\em Analytic functions of several complex variables. }
Reprint of the 1965 original. AMS Chelsea Publishing, Providence, RI, 2009. 

\bibitem{jarpfl}
M. Jarnicki and P. Pflug, \textit{
On automorphisms of the symmetrized bidisc.} 
Arch. Math. (Basel) 83 (2004), no. 3, 264--266. 
\bibitem{narasimhan}
R. Narasimhan, \textit{Several Complex Variables}, Chicago Lectures in Mathematics, University of Chicago Press, 1971.
\bibitem{rudin}
W. Rudin, \textit{Function Theory in the unit ball of $\C^n$}, Grundlehren der mathematischen Wissenschaften 241, Springer-Verlag, 1980.
\bibitem{whitney}
H. Whitney, \textit{Complex Analytic Varieties}, Addison-Wesley Series in Mathematics, Addison-Wesley, 1972.




 \end{thebibliography}
\end{document}